\documentclass[a4wide,11pt,english]{article}
\usepackage{a4wide}
\usepackage{amssymb}
\usepackage[all]{xy}
\usepackage{amsthm}
\usepackage{mathdots}
\usepackage{amsmath}
\usepackage{verbatim}
\usepackage{pifont}
\usepackage{mathrsfs}
\usepackage[ansinew]{inputenc}
\usepackage{hyperref}
\newtheorem{thm}{Theorem}[section]
\newtheorem{lemma}[thm]{Lemma}
\newtheorem{defi}[thm]{Definition}

\newtheorem{cor}[thm]{Corollary}

\newtheorem{prop}[thm]{Proposition}

\newcommand{\Aut}{\operatorname{Aut}}
\newcommand{\aut}{\operatorname{\mathfrak{aut}}}
\newcommand{\id}{\operatorname{id}}

\newcommand{\rank}{\operatorname{rank}}
\newcommand{\diag}{\operatorname{diag}}
\newcommand{\adiag}{\operatorname{adiag}}

\newcommand{\dist}{\operatorname{dist}}
\newcommand{\op}{\operatorname{op}}
\newcommand{\vp}{{\varphi}}
\newcommand{\eps}{{\varepsilon}}
\newcommand{\tr}{\operatorname{tr}}

\newcommand{\cR}{{\mathcal R}}

\newcommand{\cL}{{\mathcal L}}

\newcommand{\C}{{\mathbb C}}
\newcommand{\N}{{\mathbb N}}

\newcommand{\set}[2]{\left\{#1\left|\; #2\right.\right\}}
\newcommand{\rest}[2]{#1\raisebox{-0.5ex}{\mbox{$\mid_{#2}$}}}

\newcommand{\bt}{\begin{thm}}
\newcommand{\et}{\end{thm}}
\newcommand{\bp}{\begin{proof}}
\newcommand{\ep}{\end{proof}}
\newcommand{\bl}{\begin{lemma}}
\newcommand{\el}{\end{lemma}}
\newcommand{\bc}{\begin{cor}}
\newcommand{\ec}{\end{cor}}
\newcommand{\bd}{\begin{description}}
\newcommand{\ed}{\end{description}}

\newcommand{\fA}{{\mathfrak A}}

\newcommand{\fC}{{\mathfrak C}}

\newcommand{\fT}{{\mathfrak T}}

\newcommand{\fk}{\mathfrak k}

\newcommand{\fm}{\mathfrak m}

\newcommand{\M}{\mathbb M}
\newcommand{\Z}{\mathbb Z}
\newcommand{\Li}{\mathbb L}

\newcommand{\sub}{\subseteq}

\begin{document}

\thispagestyle{empty}
\title{\textsf{Inductive limits of finite dimensional hermitian symmetric spaces and K-theory}}
\date{}
\author{Dennis Bohle, Wend Werner\thanks{The second named author acknowledges the support of Deutsche Forschungsgemeinschaft through SFB 878 as well as the Australian Research Council.
He thanks colleagues at the Mathematical Sciences Institute in the Australian National University
for hospitality while this research was completed.}\\
Mathematisches Institut\\
Westf\"alische Wilhelms-Universit\"at\\
Einsteinstra\ss e 62\\
48149 M\"unster\\
\\
dennis.bohle@math.uni-muenster.de\\
wwerner@math.uni-muenster.de}
\maketitle

\begin{abstract}
K-Theory for hermitian symmetric spaces of non-compact type, as developed recently by the authors,
allows to put Cartan's classification into a homological perspective. We apply this method
to the case of inductive limits of finite dimensional hermitian symmetric spaces. This might be seen as an indication of how
much more powerful the homological theory is in comparison to the more classical approach.

When seen from high above, we follow the path laid out by a similar result in the theory of C*-algebras.
Important is a clear picture of the behavior of morphisms between bounded
symmetric domains of finite dimensions, which is more complex than in the C*-case, as well as an accessible
K-theory.
We furthermore have to slightly modify the invariant from our previous work.
Roughly, we use traces left by co-root lattices on K-groups, instead of co-roots themselves, which had been used previously.
\end{abstract}

\section{Introduction}
In \cite{BoWe2}, the classification for hermitian symmetric spaces of non-compact type
was shown to possess a homological background, coming from appropriately defined ordered K-groups,
with some marked subsets. We show here that this new approach permits to classify inductive limits
of finite dimensional bounded symmetric domains in quite the same way.

The category in which the inductive limits are built are the bounded symmetric domains of complex
Banach spaces. Hence, limits are obtained by an apparently unavoidable completion of the algebraic limit.
Classifying symmetric spaces by the classical method has two important ingredients: Root systems and
a Hilbertian structure on the ambient vector space (the Killing form). Both these structures become
in general impossible to deal with when tackling inductive limits. The Hilbertian structure disappears
and morphisms of higher multiplicity do not respect (simple) roots. They do respect root lattices, however.

In our approach, a threefold product will play a crucial role. It is defined on the Banach space
surrounding the domain and turns it into an algebraic object, a JB*-triple system, more general
than a C*-algebra.
It is intimately linked to the underlying Lie algebra of the
holomorphic automorphisms of the domain, so that e.g. roots still leave their traces (in the form of
so called \emph{grids}), as explored in the work of Neher \cite{Neher-3gradedrootsystemsandgridsinJordantriplesystems,
Neher-Systemesderacines3gradues}. The approach in \cite{BoWe2} favored an approach in which (classes coming from) grids
where marked in K-groups. Here, we will follow a slightly different path, laid out in \cite{BoWe3}, which
substitutes roots by root lattices. Furthermore, the natural norm of a JB*-triple system behaves favorable in the process
of building inductive limits, thus surmounting the intrinsic difficulty posed by the Hilbertian structure in the
Cartan/Serre approach.

As JB*-triple systems resemble in many aspects C*-algebras, this paper might be seen as a non-trivial
extension of the very important work of Elliot \cite{Elliott-Ontheclassificationofinductivelimitsofsequencesofsemisimplefinitedimensionalalgebras},
which has influenced much of the modern structure theory of C*-algebras.

Our main result is

\subparagraph{Theorem}
\emph{Each inductive limit of finite dimensional bounded symmetric domains $U$ in a complex Banach space
$E$ admits a decomposition
\begin{equation*}
    U=U_P\oplus U_E\oplus U_S\oplus U_H
\end{equation*}
where all of the summands are sums of open unit balls; for the summand $U_H$, of Hilbert spaces,
of spin factors in case of the domain $U_S$, and of exceptional factors for $U_E$.
\\
The principal part $U_P$ of $U$ is an inductive limit
of indecomposable hermitian, symplectic and rectangular domains and can be classified by its $K^{\pm}$-invariant.}

\subparagraph{}
Unintelligible notation used in the above will be explained in the next sections.

Our plan here ist the following: We start (in section 3) with an investigation of morphisms
between finite dimensional bounded symmetric domains. Not all these morphisms will eventually play a role in the
inductive limit, and so we restrict our attention to those which do. In section 4 an
invariant is introduced, based on K$_0$-groups carrying some marked subsets together with one of its involutive
automorphisms.
Section 5 contains the proof of the fact that we are dealing with a complete isomorphism invariant, and
in the final section we modify the concept of a Brateli
diagram, a graphical storage device used in C*-algebra theory, in order to become helpful for the
actual classification of inductive limits.

We will briefly collect necessary technical background in the following section.

\section{Background}
Standard references for the following are the monographs \cite{BlecherLeMerdy-Operatoralgebrasandtheirmodules,ChuBook,
Upmeier-SymmetricBanachmanifoldandJordaCalgebras}, the surveys \cite{Russo-Survey,Angel-Survey} as well as the articles
\cite{Harris-Boundedsymmetrichomogeneousdomainsininfinitedimensionalspaces,
Kaup-ARiemannMappingTheoremForBoundedSymmetricDomainsInComplexBanachSpaces}.

\paragraph{Bounded symmetric domains}
A bounded domain $\Omega$ in a complex Banach space $E$ is called symmetric iff for every $x\in\Omega$, there is an
involutive biholomorphic map of $\Omega$ onto itself for which $x$ is an isolated fixed point.
In finite dimensions, each such domain can be equipped with an essentially unique hermitian
structure so that biholomorphic automorphism coincide with isometries. Finite dimensional
bounded symmetric domains turn out to be the hermitian symmetric spaces of non-compact type.

Each bounded symmetric domain can actually be biholomorphically mapped onto the open unit ball $U$ of
a Banach space $E$. Let $\Aut U$ denote the group of biholomorphic automorphism of $U$ and denote by $K$
the isotropy subgroup at the origin. Then $U=\Aut U/K$. $\Aut U$ is a (real) Lie group, and if we pass to its
Lie algebra $\aut U$, the symmetry of $U$ yields a $\Z_2$-grading $\aut U=\fk\oplus\fm$, where $\fk$ is a Lie subalgebra,
and $\fm$ can be identified with the Banach space $E$ itself. $\fm$ is not a Lie subalgebra; due to the grading,
however, it has the property that for any choice of elements $m_{1,2,3}\in\frak{m}$ the triple Lie bracket
$[m_1,[m_2,m_3]]$ is an element of $\frak{m}$.

\paragraph{JB*-triple systems}
A very prominent example of a bounded symmetric domain is
the open unit ball of a C*-algebra $\mathfrak{A}$. In this case, the Lie triple product is connected
to the C*-structure by the equation
\begin{equation*}
    [a,[b,c]]-i[a,[ib,c]]=\frac{ab^\ast c+cb^\ast a}{2}
\end{equation*}
valid for all $a,b,c$ in $\mathfrak{A}$. This leads to the following definition: A closed subspace $Z$ of a C*-algebra
$\frak{A}$ is called a \emph{JC*-triple}
iff for all $a,b,c\in Z$
\begin{equation*}
    \{a,b,c\}:=\frac{ab^\ast c+cb^\ast a}{2}\in Z.
\end{equation*}
The open unit ball of a JC*-triple is a bounded symmetric domain, but not all these
domains appear in such a way. A slightly more general, and
intrinsic, definition leads to all of them.
\begin{defi}
A Banach space $Z$ together with a trilinear, continuous and w.r.t. the outer
variables symmetric mapping $\{\cdot,\cdot,\cdot\}:Z^3\to Z$ is called a \emph{JB*-triple}, iff
\begin{itemize}
\item[(a)] $\|\{x,x,x\}\|=\|x\|^3$ for all $x\in Z$,
\item[(b)] With the operator $x\Box y$ defined on $Z$ by $(x\Box y)(z)=\{x,y,z\}$,
$ix\Box y$ is a derivation for the triple product,
\item[(c)] For each $x\in Z$, $x\Box x$ has non-negative spectrum, and $\exp(it(x\Box x))$ is a 1-parameter group of isometries.
\end{itemize}
\end{defi}
Up to biholomorphic maps, bounded symmetric domains and the open unit balls of JB*-triple are now the
same class\cite{Kaup-AlgebraiccharacterizationofsymmetriccomplexBanachmanifolds}.
In finite dimensions, the open unit balls of the following JB*-triples
represent the classical, indecomposable, bounded symmetric domains, the \emph{Cartan factors}:
\begin{description}
\item[Type I] complex $n\times m$-matrices $\M_{n,m}(\mathbb{C})$, the \emph{rectangular factor} $C^{\textrm{I}}_{m,n}$
\item[Type II] skew-symmetric, complex $n\times n$-matrices, the \emph{symplectic factor}, $C^{\textrm{II}}_n$, $n\geq 4$,
\item[Type III] the \emph{hermitian factor}, consisting of symmetric complex $n\times n$-matrices, $C^{\textrm{III}}_n$, $n\geq 2$,
\item[Type IV] the $n$-dimensional \emph{spin factor} $C^{\textrm{IV}}_n$, which is the
closed linear span of selfadjoint matrices $s_1,\ldots,s_n$, $n\geq 2$, satisfying
\begin{equation*}
s_is_j+s_js_i=2\delta_{ij}1
\end{equation*}
for all $i,j\in\{1,\ldots,n\}$
\item[Type V,VI] two \emph{exceptional factors} in dimensions $16$ and $27$
\end{description}
The triple product is, for the non-exceptional factors, the one based on matrix
products (and standard conjugation). These factors are mutually non-isomorphic with the exception
of $\M_{2,2}$, the symplectic factor for $n=4$, and the hermitian factor for $n=2$, which are spin factors,
of dimensions 4, 6 and 3, respectively. Note that the vector space structure admits the identity as an
element of $C^{\textrm{IV}}_n$ only if $n=1$.

JB*-triples that cannot be embedded into a C*-algebra have this property due
to the fact that they possess quotients of type V or VI, respectively.

Recall that type IV Cartan factors $Z$ can, alternately, be defined as the closed, self-adjoint subspace
of the space $L(H)$ of all bounded operators on the Hilbert space $H$, such that
$Z^2\sub\C\id$. By polarization, there exists a scalar product $\langle\cdot,\cdot\rangle$ on $Z$ so that
\begin{equation*}
    AB^\ast+B^\ast A=2\langle A,B\rangle\id,
\end{equation*}
which yields a canonical Hilbert space structure on each spin factor $Z$. At the same time, $Z$
carries an involution, which will be called `canonical'.

An element $u$ of a JB*-triple $Z$ is called \emph{tripotent} iff $\{u,u,u\}=u$.
Two elements $u,v\in Z$ are called \emph{orthogonal}, iff
$\{a,a,b\}=0$, which is the same as $a^\ast b=ab^\ast =0$ if E is a JC*-subtriple of a C*-algebra.
The \emph{rank} of a JB*-triple system is the maximal
cardinality of a system of non-zero orthogonal tripotents it contains. The rank is an
isomorphism invariant and if $\vp:Z\to W$ is a morphism of
JB*-triple systems, then, necessarily,
$$\rank\vp(Z)\leq\rank W.$$
The ranks of the different Cartan
factors are well known. If $Z\simeq \M_{n,m}$ is a type I Cartan
factor, then $\rank Z=\min\{n,m\}$. The symplectic Cartan factors
$C^{\textrm{II}}_{2n}$ and $C^{\textrm{II}}_{2n+1}$ have $\rank$ $n$ as well as the
hermitian Cartan factor $C^{\textrm{III}}_n$. All Spin factors, independent of
their dimension, have $\rank$ $2$.

\paragraph{Universal TRO}\label{Universal TRO}
A left (full) Hilbert C*-Module $(E,\langle\cdot,\cdot\rangle_\ell)$ with $\langle\cdot,\cdot\rangle_\ell$
taking values in the left C*-algebra $\mathcal{L}(E)$, is, in a natural
way, a JB*-triple:
If $\langle\cdot,\cdot\rangle_r$ denotes the canonical form, acting on the right of $E$ and
taking values in the right C*-algebra $\mathcal{R}(E)$, then
\begin{equation*}
    \{x,y,z\}:=\frac{\langle x,y\rangle_\ell z+x\langle y,z\rangle_r}{2}
\end{equation*}
turns $E$ into a JB*-triple. What distinguishes Hilbert C*-modules among JB*-triples is the existence
of an additional, associative ternary product. More precisely, a JB*-triple $E$ belongs to this class iff it embeds
into a C*-algebra, satisfying the more restrictive condition
\begin{equation*}
    [x,y,z]:=xy^\ast z\in E \quad\text{for all $x,y,z\in E$}.
\end{equation*}
Note that in such a representation, which deviates somewhat from the usual one, left and right C*-algebras of a
Hilbert C*-module can be identified with the closed linear span of the operators
\begin{equation*}
x\longmapsto zy^\ast x,\qquad\text{resp.}\qquad x\longmapsto xy^\ast z.
\end{equation*}
Objects in a category of Hilbert C*-modules in which morphisms are required to respect the product $[\cdot,\cdot,\cdot]$
are usually called \emph{Ternary Rings of Operators}, or TROs, for short.
For each TRO $T$ there exists a canonical C*-algebra, the \emph{linking algebra $\Li(T)$},
formally defined through
$$
\Li(T)=\begin{pmatrix}\cR(T)& T^\ast \\
                        T &\cL(T)\end{pmatrix},
$$
with $T^\ast $ the TRO conjugate to $T$.
Note that the expressions $\cL,\cR$, and $\Li$ actually become functorial by letting
\begin{gather*}
\cL(\Phi)(\sum x_iy_i^\ast )=\sum \Phi(x_i)\Phi(y_i)^\ast \qquad
\cR(\Phi)(\sum x_i^\ast y_i)=\sum \Phi(x_i)^\ast \Phi(y_i)\\
\text{and}\qquad
\Li(\Phi)=\begin{pmatrix}\cR(\Phi)& \Phi^\ast \\
                        \Phi &\cL(\Phi)\end{pmatrix}
\end{gather*}
for any TRO-morphism $\Phi$ (see \cite{Hamana-TripleenvelopesandSilovboundariesofoperatorspaces} for details).

In \cite{BoWe1} and \cite{BunceFeelyTimoneyI} a \emph{universal (enveloping) TRO} $\tau(Z)$ has been assigned
to each JB*-triple $Z$, as well as a canonical embedding $\rho_Z:Z\to \tau(Z)$. The image $\rho_Z(Z)$ generates
$\tau(Z)$ as a TRO.
Furthermore, each JB*-morphism uniquely lifts to a TRO-morphism between the respective universal TROs,
and the emerging functor has all properties needed for the ensuing K-theory.
It is important that for a finite dimensional $Z$, the universal TRO can be calculated. The result in the
case of the Cartan factors is as follows.
\begin{itemize}
\item
For the rectangular factor $\M_{m,n}(\mathbb{C})$, $m,n\geq 2$,
the universal TRO is $\M_{m,n}(\mathbb{C})\oplus \M_{n,m}(\mathbb{C})$,
with embedding $A\mapsto A\oplus A^\top$.
\item
Though of rectangular type, Hilbert spaces $H$, of dimension $N$,
have to be treated separately. Here,
\begin{equation*}
\tau(H)=\bigoplus_{n=0}^{N-1} L(\bigwedge^{n+1}H,\bigwedge^n H),
\end{equation*}
and $h\mapsto a(\overline{h})$, with $a$ being the annihilation operator,
is the canonical embedding.
\item
For both, the hermitian and symplectic factors $Z$ of $\M_N(\mathbb{C})$,
their natural embedding into $\M_N(\mathbb{C})$ represents the map
$\rho_Z$.
\item
In the case of the spin factor $Z$ of dimension $n$, we have
$\tau(Z)=\M_{2^n}(\mathbb{C})$ if $n$ is even,
$\tau(Z)=\M_{2^n}(\mathbb{C})\oplus \M_{2^n}(\mathbb{C})$ otherwise.
The embedding is well known from the construction of the CAR-algebra.
\item
In the case of the exceptional factors the universal TRO is the null space.
\end{itemize}
We will make use of a uniquely determined, involutive antiautomorphism $\Phi:\tau(Z)\to \tau(Z)$,
which is the identity on (the canonical copy of) $Z$ \cite[Theorem 3.5]{BunceFeelyTimoneyI}. It can
be defined as follows. Denote, for a given TRO $X$ the opposite TRO by $X^{\op}$, for which $X=X^{\op}$
as complex vector spaces and
\begin{equation*}
    [x,y,z]^{\op}=[z,y,x],\qquad x,y,z\in X^{\op}.
\end{equation*}
Let $\op_X: X\to X^{\op}$ be the formal identity and note that there is a canonical JB*-embedding
$\rho^{\op}_Z:Z\to \tau(Z)^{\op}$. In the present situation, as
$\tau(Z)$ is universal, there exists a TRO-isomorphism $\Phi_0:\tau(Z)\to\tau(Z)^{\op}$
such that $\Phi_0\rho_Z=\rho_Z^{\op}$. Then $\Phi=\op_{\tau(Z)}^{-1}\Phi_0$ is the canonical
involutive antiautomorphism. (Note that $\Phi^2$ is a linear automorphism on $\tau(Z)$ which is the identity
on $Z$.)

This construction has its origin in a similar one from the theory of Jordan algebras. We refer
to \cite[Theorem 7.1.8]{HancheOlsenStoermerJordanoperatoralgebras} for more details.
For hermitian (symplectic) factors, $\Phi$ maps a matrix in the universal TRO
onto its (negative) transpose. On $\M_{m,n}\oplus \M_{n,m}$, the universal
TRO of the rectangular factor $\M_{m,n}$, we have $\Phi(A\oplus B)=B^\top\oplus A^\top$,
which is essentially what happens in the case of Hilbert spaces as well.

JC*-triples $Z$ of types I--III distinguish themselves by the fact that
\begin{equation*}
    \rho_Z(Z)=\set{z\in\tau(Z)}{\Phi(z)=z}
\end{equation*}
Actually, this behavior is characteristic for so called \emph{universally reversible JC*-triples}.
See \cite{BunceFeelyTimoneyI} for details.
\paragraph{K-theory}
The K-functor we will use in the following is obtained in two steps. The first defines
K-groups for TROs: For a TRO $T$ we let $K_\ast^{\mathrm{TRO}}(T)$ be the K-group of its left C*-algebra.
(Note that the right C*-algebra of $T$ is Morita equivalent, and there is no difference
as to which one of them is selected.) As each TRO-morphism $\psi$ yields
C*-morphism $\mathcal{L}(\psi)$ and $\mathcal{R}(\psi)$ between both, left and right C*-algebras,
$K_\ast^{\mathrm{TRO}}(\psi)$ is defined in a functorial way.
\begin{defi}
Denote by $\tau$ the functor that provides each JB*-triple with its universal
TRO. Define, for a JB*-triple Z and a JB*-morphism $\phi$,
\begin{equation*}
    K_\ast^{\mathrm{JB}^\ast }(Z)=K_\ast^{\mathrm{TRO}}(\tau(Z)),
\end{equation*}
as well as
\begin{equation*}
    K_\ast^{\mathrm{JB}^\ast }(\phi)=K_\ast^{\mathrm{TRO}}(\tau(\phi))
\end{equation*}
\end{defi}
This functor has the usual properties that one would expect from it,
except stability, which already has a bad start, as it is not possible,
in general, to equip the space of matrices with entries from a JB*-triple
with the structure of a JB*-triple in a natural way.

The $K_0$-groups for the Cartan factors are the following.
\begin{itemize}
\item $\mathbb{Z}^2$ in the case of rectangular $m\times n$-Matrices, $m,n\geq 2$,
\item $\mathbb{Z}^N$ for Hilbert spaces of dimension $N$,
\item $\mathbb{Z}$ for both, hermitian and symplectic factors,
\item $\mathbb{Z}$ in case of odd dimensional spin factors, and
$\mathbb{Z}^2$ for the even dimensional spin factors)
\end{itemize}

\section{Classifying JB*-morphisms}
In this section, morphisms between finite dimensional
Cartan factors will be classified, in as much as this will be needed in the sequel.
The technique applied here is good for more, though.
We will use an equivalence relation for morphisms which yields identical maps
on K-groups.

\begin{defi}
Suppose $R$ and $S$ are TROs, and let $\phi_{1,2}:R\to S$ be TRO-morphisms.
\begin{description}
\item[(a)]
We call $\phi_1$ and $\phi_2$ \emph{unitarily equivalent} iff
there are unitaries $U\in\mathcal{R}(S)$ and $V\in\mathcal{L}(S)$ so that $\phi_1(x)=V\phi_2(x)U$ for all $x\in R$.
\item[(b)]
JB*-morphisms $\vp_{1,2}:Z_1\to Z_2$ between finite dimensional JC*-triple systems will be called unitarily
equivalent iff the TRO-morphisms $\tau(\vp_1)$ and $\tau(\vp_2)$ have this property.
\end{description}

\end{defi}

\begin{prop}[cf.\ \cite{Bohle}, Proposition 3.2.5]\label{matrixdarst von Tromorph}
Let $R$ and $S$ be finite dimensional TROs and $\phi:R\to S$ be
a TRO-morphism. Then, up to
unitary equivalence, $\phi$ is uniquely determined by a matrix
with entries in $\N_0$.

More precisely, suppose $R=\bigoplus_{i=1}^{p}
\mathbb{M}_{n_i,m_i}$, $S=\bigoplus_{j=1}^q
\mathbb{M}_{k_j,l_j}$ and let $\varphi^j:R\to \M_{k_j,l_j}$ be the compression of $\varphi$ to $\M_{k_j,l_j}$ for $j\in\{1,\ldots,q\}$ so that $\vp=\vp^1\oplus\cdots\oplus\vp^q$. Then there is a  $q\times p$-matrix
$(\alpha_{i,j})$ with entries in $\mathbb{N}_0$ with
$\sum_{j=1}^p\alpha_{i,j}n_j\leq k_i$ and $\sum_{j=1}^p\alpha_{i,j}m_j\leq l_i$
and such that each $\vp^j$ is unitarily equivalent to $\id_1^{\alpha_{1,j}}\oplus\ldots\oplus \id_p^{\alpha_{p,j}}$, where $\id_i$ denotes the identity representation of
$\mathbb{M}_{n_i,m_i}$ for $i=1,\ldots,p$.
\end{prop}
For the convenience of the reader, we include a proof. We first treat the case of a morphism $\vp:Z_1\to Z_2$
between TRO-factors $Z_{1,2}$.
For the factor $Z_2=\M_{m,n}$, the linking algebra equals $\M_N$, $N=m+n$.
By a well-known result, there is a unitary $U_0\in\M_N$ so that
\begin{equation*}
    \Li\vp\begin{pmatrix}s^\ast s & x^\ast\\ y & tt^\ast\end{pmatrix}=
    \begin{pmatrix}\cR\vp(s^\ast s) & \vp^\ast(x^\ast)\\ \vp(y) & \cL\vp(tt^\ast)\end{pmatrix}=
    U_0\left(\diag_k\left[\begin{pmatrix}s^\ast s & x^\ast\\ y & tt^\ast\end{pmatrix}\right]\oplus 0_{N-2k}\right)U_0^\ast
\end{equation*}
For an appropriately defined permutation matrix $\Pi$,
\begin{equation*}
    \diag_k\begin{pmatrix}s^\ast s & x^\ast\\ y & tt^\ast\end{pmatrix}\oplus 0_{N-2k}=
    \Pi\begin{pmatrix}\diag_k(s^\ast s)\oplus 0_{N-2k}& \diag_k(x^\ast)\oplus 0_{N-2k}\\
                   \diag_k(y)\oplus 0_{N-2k} & \diag_k(tt^\ast)\oplus 0_{N-2k}\end{pmatrix}\Pi^{-1}
\end{equation*}
Let $U=\Pi U_0=\left(\begin{smallmatrix}U_{11} & U_{12}\\ U_{21} & U_{22}\end{smallmatrix}\right)$,
with $U_{11}\in\M_n$, $U_{22}\in\M_m$, $U_{12}\in\M_{nm}$, and $U_{21}\in\M_{mn}$.
Checking dependencies on the variables in
\begin{equation*}
    \begin{pmatrix}\cR\vp(s^\ast s) & \vp^\ast(x^\ast)\\ \vp(y) & \cL\vp(tt^\ast)\end{pmatrix}=
    U\begin{pmatrix}\diag_k(s^\ast s)\oplus 0_{N-2k}& \diag_k(x^\ast)\oplus 0_{N-2k}\\
                   \diag_k(y)\oplus 0_{N-2k} & \diag_k(tt^\ast)\oplus 0_{N-2k}\end{pmatrix}U^\ast
\end{equation*}
shows $U_{12}=U_{21}=0$, and the result follows in this special case. To complete the proof,
it is enough to consider the case in which $q=1$. We may suppose that $\phi$ is injective.
Then its image is a direct sum of sub-TROs $S_i=\phi(\mathbb{M}_{n_i,m_i})$, $i=1,\ldots,p$ of $S$.
Applying the above we find unitaries $V_i\in\mathcal{L}(S_i)$ and $U_i\in\mathcal{R}(S_i)$ so that
$V_0=V_1\oplus\cdots\oplus V_p$ and $U_0=U_1\oplus\cdots\oplus U_p$ provide unitary equivalence
on the sub-TRO level. But each unitary in the left (right) C*-algebra of the sub-TRO of a finite
dimensional TRO can be extended to the whole, and the result in case $q=1$ follows. The general
case now follows in an obvious way.
\begin{defi}
Let $R,S$ be finite dimensional TROs and $Z,W$ finite dimensional JC*-triples.
\begin{description}
\item[(a)]
For a morphism $\vp:R\to S$ we call the matrix
$(\alpha_{ij})\in \M_n(\N_0)$ representing $\vp$ its \emph{multiplicity (matrix)}.
\item[(b)]
Similarly, the \emph{multiplicity (matrix)} of a morphism $\psi:Z\to W$ is the matrix
$(\beta_{ij})\in \M_n(\N_0)$ representing the TRO-morphism $\tau(\psi)$.
\end{description}
\end{defi}

It should be noted that according to the definition of unitary equivalence,
the multiplicity matrices of $\tau(\psi)$ are the same
as those of $\cL(\tau(\psi))$ and $\cR(\tau(\psi))$.

The morphisms that will be important in the following are essentially of two types.
\begin{defi}
We call a morphism between JB*-factors of types I-III of
class $(A)_{k,\ell}$, $k,\ell\in\N$, iff, up to unitary equivalence, it is of the form
\begin{equation*}
    A\longmapsto
    \begin{pmatrix}
     &\diag_k(A) & &0          &\\
     &           &0&           &\\
     &0          & &\diag_\ell(A^\top) &
    \end{pmatrix}
\end{equation*}
and of class $(B)_{k,\eps}$, $\eps=\pm 1$, iff it is unitarily equivalent to
\begin{equation*}
    A\longmapsto
    \begin{pmatrix}
     & 0                   & &\adiag_k(A)&\\
     &                     &0&           &\\
     &\adiag_k(\eps A^\top)& &        0  &
    \end{pmatrix}, \qquad \eps\in\{\pm 1\}
\end{equation*}
where the matrix $\adiag_k(A)$ displays $k$ matrices $A$ on its off-diagonal and is zero
elsewhere.
\end{defi}
In the following theorem we do not treat morphisms mapping
factors of low rank --- Hilbert spaces, spin factors --- into factors of high
rank, and different type, as they do not play a role in the investigation of inductive limits.

\begin{thm}\label{MorphClass}
Let $\vp:Z\to W$ be a non-zero JB*-morphism between Cartan factors $Z$ and $W$.
Then, up to unitary equivalence, the following cases occur.
\begin{description}
\item[(i)]
  If $\vp$ acts between finite dimensional factors of type I-III, where in case of a factor
  of type I Hilbert spaces must be excluded, it is of type
\begin{description}
\item[(A)]
    $(A)_{k,\ell}$ in case $I\to I$ and of type $(A)_{k,0}$ in the remaining 4 cases, $II\to II$, $III\to III$,
    $II\to I$ and $III\to I$.
\item[(B)]
    $(B)_{k,-}$ in case $I\to II, III\to II$, and of type $(B)_{k,+}$ for the cases $I\to III, II\to III$.
\end{description}
Furthermore, the following multiplicities can occur ($\alpha,\beta\in\N$):
\begin{equation*}
\begin{array}{|c|c|c|c|}
  \hline
                      &                      &                  &                  \\
  \longrightarrow     & C^{\textrm{I}}_{M,N} & C^{\textrm{II}}_N&C^{\textrm{III}}_N\\
                      &                      &                  &                  \\ \hline
                      &                      &                  &                  \\
  C^{\textrm{I}}_{m,n}&\begin{pmatrix}\alpha &\beta\\ \beta &\alpha\end{pmatrix}
                                             &\begin{pmatrix}\alpha &\alpha\end{pmatrix}
                                                                &\begin{pmatrix}\alpha &\alpha\end{pmatrix}\\
                      &\alpha m+\beta n\leq M,
                       \alpha n+\beta m\leq N&\alpha(m+n)\leq N &\alpha(m+n)\leq N \\
                  &                      &                  &                  \\ \hline
                  &                      &                  &                  \\
  C^{\textrm{II}}_n&\begin{pmatrix}\alpha\\ \alpha\end{pmatrix}
                                         &\begin{pmatrix}\alpha\end{pmatrix}
                                                            &\begin{pmatrix}2\alpha\end{pmatrix}\\
                  &\alpha n\leq\min\{M,N\}                  & \alpha n\leq N   & 2\alpha n\leq N  \\
                  &                      &                  &                  \\ \hline
                  &                      &                  &                  \\
  C^{\textrm{III}}_n&\begin{pmatrix}\alpha\\ \alpha\end{pmatrix}
                                         &\begin{pmatrix}2\alpha\end{pmatrix}
                                                            &\begin{pmatrix}\alpha\end{pmatrix}\\
                  &\alpha n\leq\min\{M,N\}                  & 2\alpha n\leq N  & \alpha n\leq N  \\
                  &                      &                  &                  \\ \hline
\end{array}
\end{equation*}
\item[(ii)]
  The morphisms between
  \begin{description}
    \item[(a)] Hilbert spaces
    \item[(b)] spin factors
  \end{description}
  are all of multiplicity one. More precisely, in the Hilbertian case, they
  coincide with the isometric embeddings, whereas for a non-zero morphism between spin factors,
  $\vp:Z_1\to Z_2$, there exists an isometric embedding $U:Z_1\to Z_2$ for the canonical Hilbert space structures of $Z_1$ and $Z_2$ with $U(x^\ast )=U(x)^\ast $ for all $x\in Z_1$ and a complex number $\mu\in\C$ with $|\mu|=1$ such that
  \begin{equation*}
  \vp(x)=\mu U(x)\qquad\text{for all $x\in Z_1$.}
  \end{equation*}
\end{description}

\end{thm}
The fact that exchanging type II with type III factors leaves the respective results (as, essentially, their proofs)
unchanged is no accident and can be explained by the fact that there is an involutive equivalence of these
subcategories (to be explained in more detail in the next section).

\subsection{Proof of Theorem \ref{MorphClass}}

\paragraph{Case $\mathbf{I\to I}$}
Let $Z$ and $W$ be type I Cartan factors with $\rank Z$, $\rank
W\geq 2$, and suppose that $Z$ and $W$ are embedded in their
universal universal TROs
\[Z=\{(A,A^\top):A\in \M_{n,m}\}\sub \tau(Z)=\M_{n,m}\oplus\M_{m,n}\]
and
\[W=\{(B,B^\top):B\in \M_{N,M}\}\sub \tau(Z)=\M_{N,M}\oplus\M_{M,N}.\]
We know from Proposition \ref{matrixdarst von Tromorph} that the mapping $\tau(\vp):\tau(Z)\to \tau(W)$ is uniquely, up to unitary equivalence, determined by a $2\times 2$ matrix
$\begin{pmatrix}
  \alpha & \beta\\
  \gamma & \delta
\end{pmatrix}$
with entries in $\N_0$. Moreover, by the same result,
\begin{align}\label{Schranke fuer type I}
  &\alpha n+\beta m\leq N,
  &\gamma n+\delta m\leq M,\\
\nonumber  &\alpha m+\beta n\leq M,
  &\gamma m+\delta n\leq N.
\end{align}
and there exist unitaries $U_1,U_2$ and $K_1,K_2$ such that
$$
\tau\left(\vp\right)=
\left(U_1\left(\id_{n,m}^\alpha\oplus\id_{m,n}^\beta\right)K_1,U_2\left(\id_{n,m}^\gamma\oplus\id_{m,n}^\delta\right)K_2\right),
$$
where $\id_{j,k}^\gamma$ denotes the $\gamma$-fold identity
representation of $\M_{j,k}$. Now $\tau\left(\vp\right)\left(Z\right)\sub W$ and therefore
$$
\left(U_1\left(\id_{n,m}^\alpha\left(A\right)\oplus\id_{m,n}^\beta\left(A^\top\right)\right)K_1\right)^\top=
U_2\left(\id_{n,m}^\gamma\left(A\right)\oplus\id_{m,n}^\delta\left(A^\top\right)\right)K_2,
$$
for all $A\in \M_{n,m}$. This yields
\begin{align*}
  \id_{n,m}^\gamma\left(A\right)\oplus\id_{m,n}^\delta\left(A^\top\right)&=
  U_2^\ast K_1^\top\left(\id_{m,n}^\alpha\left(A^\top\right)\oplus\id_{n,m}^\beta\left(A\right)\right)U_1^\top K_2^\ast \\
  &=\widetilde{U}\left(\id_{n,m}^\beta\left(A\right)\oplus\id_{m,n}^\alpha\left(A^\top\right)\right) \widetilde{K},
\end{align*}
for suitable unitaries $\widetilde{U}$ and $\widetilde{K}$. But,
since $Z$ generates $\tau(Z)$ as a TRO, this can only be
true if $\alpha=\delta$ and $\beta=\gamma$. On the other hand, each of the TRO-morphisms with this multiplicity
stems from a JC*-morphism, uniquely determined up to unitary equivalence.

\paragraph{Case $\mathbf{II,III\to I}$}
If $Z$ is of type II (resp.\ type III ), and $W=C^{\textrm{I}}_{N,M}$ rectangular,
$N,M\geq2$, then $\tau(\vp):\tau(Z)\to\tau(W)$ is, up to unitary equivalence, uniquely determined by a
$2\times 1$-matrix
$\left(\begin{smallmatrix}
  \alpha\\
  \beta
\end{smallmatrix}\right).$
In order that $\tau(\vp)Z\sub W$, only $\alpha=\beta$ is possible, and $0\leq n\alpha\leq \min\{M,N\}$.
But each of these mappings comes from a morphism of the underlying JB*-triples.

\paragraph{Case $\mathbf{I\to II,III}$}
We treat the case $I\to III$ only; the missing case is proven almost identically.
Suppose that
\[
Z=\{(A,A^\top):A\in \M_{n,m}\}\sub
\tau(Z)=\M_{n,m}\oplus\M_{m,n},
\]
for suitable $n,m\in\N$, and, respectively,
\[
W=\{A\in\M_N:A^\top=A\}\sub \tau(W)=\M_N,
\]
for suitable $N\in\N$. If $\vp:Z\to W$ is a JB*-triple morphism,
then $\tau(\vp)$ is, up to unitary
equivalence, uniquely determined by a $1\times2$ matrix
$\begin{pmatrix}\alpha & \beta\end{pmatrix}$ with $\alpha,\beta\in\N_0$
and $0\leq\alpha n+\beta n\leq N$. Consequently, there exist
unitary matrices $U$ and $K$ such that
\begin{equation*}
\tau(\vp)=U(\id_{n,m}^\alpha\oplus\id_{m,n}^\beta)K,
\end{equation*}
We must have
\begin{equation*}
\tau(\vp)((A,A^\top))=U\left(\diag_\alpha(A)\oplus\diag_\beta(A^\top)\oplus 0_{N-\alpha-\beta}\right)K\in W,
\end{equation*}
for all $A\in \M_{n,m}$, and so $(\tau(\vp)((A,A^\top)))^\top=\tau(\vp)((A,A^\top))$, or
\begin{equation*}
  U\left(\diag_\alpha(A)\oplus\diag_\beta(A^\top)\oplus 0_{N-\alpha-\beta}\right)K
  =\widetilde{U}\left(\diag_\beta(A)\oplus\diag_\alpha(A^\top)\oplus 0_{N-\alpha-\beta}\right)\widetilde{K}
\end{equation*}
with unitary matrices $\widetilde{U}$ and $\widetilde{K}$. Since
$\M_{n,m}\oplus \M_{m,n}$ is generated by $Z=\{(A,A^\top):A\in
\M_{n,m}\}$, we must have $\alpha=\beta$ with $\alpha(m+n)\leq N$.
For appropriately chosen permutation
matrices $\Pi_{1,2}$, and unitaries $U,V$ such that
$\tau\vp(A,B)=U\left(\diag_\alpha(A)\oplus 0_{n-2\alpha}\oplus\diag_\alpha(A)\right)V$
we have
\begin{equation*}
    \vp(A)=\tau\vp(A,A^\top)=U\Pi_1
    \begin{pmatrix}
     & 0                   & &\adiag_k(A)&\\
     &                     &0&           &\\
     &\adiag_k(A^\top)    & &        0  &
    \end{pmatrix}
    \Pi_2 V.
\end{equation*}

\paragraph{Cases $\mathbf{II\to II}$ and $\mathbf{III\to III}$}
Let $Z$ and $W$ be type II Cartan factors
realized as
\[Z=\{A:A\in \M_n, A^\top=-A\}\sub \tau(Z)=\M_n\]
and
\[W=\{B:B\in \M_N, B^\top=-B\}\sub \tau(W)=\M_N.\]
for $n,N\in\N$. If $\vp:Z\to W$ is a JB*-triple morphism, then $\vp$ is, up
to unitary equivalence, uniquely determined by its multiplicity matrix
$\begin{pmatrix}\alpha \end{pmatrix}$ with
$0\leq\alpha\leq\frac{N}{n}$ and $\alpha\in\N_0$. Note that the restriction
of the standard form of $\tau\vp$ already takes takes values in $W$.
The argument in the case $III\to III$ is identical.

\paragraph{Cases $\mathbf{III\to II}$ and $\mathbf{II\to III}$}
Let, for $n,N\in\N$, $Z$ and $W$ be realized as
\[Z=\{A:A\in \M_n, A^\top=A\}\sub \tau(Z)=\M_n\]
and
\[W=\{B:B\in \M_N, B^\top=-B\}\sub \tau(W)=\M_N,\]
respectively.
Then, any JB*-morphism $\vp: Z\to W$ is again determined by a
$1\times1$-matrix $\begin{pmatrix}\alpha\end{pmatrix}$ with $0\leq \alpha\leq\frac{N}{n}$.

We first assume that $N=\alpha n$ and fix unitary matrices $U$ and $K$ such that $\vp(A)=U\id^\alpha(A) K$ for every $A\in Z$. Applying an automorphism of $W$ of the form $\psi(B)=VBV^\top$, $V$ unitary, we may assume that $\vp=U\id^\alpha$,
for a unitary matrix $U$. As $\vp(A)$ is skew-symmetric for every symmetric matrix $A$, we find, letting $A=E_n$,
$U^\top=-U$.
Write $U=(U_{i,j})$ with $U_{i,j}\in \M_n$. Then, for every symmetric matrix $A$,
\[
(U_{i,j}A)=(-U_{i,j}A)^\top=(AU_{j,i}).
\]
Thus $U_{i,j}A=AU_{j,i}$ for every $n\times n$-matrix $A$, since the symmetric matrices generate $\M_n$ as a TRO. Consequently, $U_{i,j}=\lambda_{i,j}E_n$, where necessarily $\lambda_{i,i}=0$ and $\lambda_{i,j}=-\lambda_{j,i}$. It follows in particular that $\alpha$ is even, and so
\begin{equation*}
    \vp(A)=U\id^\alpha(A)K
    =U'\left(\adiag_{N/2}(-A)\oplus\adiag_{N/2}(A)\right)K,
\end{equation*}
for an appropriately chosen unitary $U'$.

Next we assume that $\alpha n<N$. Then there exists a unitary $U$ such that
$$
\vp(A)=U\left(\diag_k(A)\oplus 0_{N-k}\right).
$$
With a similar reasoning as before we find that
$$
U=\begin{pmatrix}
  U' &0\\
  0 & B
\end{pmatrix},$$
where $U'$ and $B$ are unitary matrices and $U'=(\lambda_{i,j}E_n)$ with $\lambda_{i,j}=-\lambda_{j,i}$. We can now easily deduce that $\alpha$ has to be even and
proceed as before.


The case $III\to II$ is very similar and therefore omitted.

\paragraph{Case $\mathbf{IV\to IV}$}
Let $Z_1$ and $Z_2$ be spin factors and denote by $\chi(Z_1)$ the selfadjoint part of $Z_1$.
We start by showing that $\vp$ is unitary with respect to the underlying Hilbertian structure. In fact,
when embedded into its universal TRO, the associated scalar product is defined through
\begin{equation*}
    ST^\ast+T^\ast S=2\langle S,T\rangle\id
\end{equation*}
and hence is the restriction to $Z$ of
\begin{equation*}
    \langle S,T\rangle=\tr(ST^\ast ),
\end{equation*}
where $\tr$ denotes the normalized trace. Then, as $\tau(\vp)$
is unitarily equivalent to a multiple of the identity map, this scalar product must be left invariant under $\tau(\vp)$,
and so $\vp$ has to be an isometric embedding with respect to the Hilbertian structure.
For the reminder of the proof we follow the method developed in \cite[Section 2]{HervesandIsidro-Isometriesandautomorphismsofthespacesofspinors}.

Since $Z_1$ and $Z_2$ are factors, $\vp$ is injective. We will first prove that there exists a complex number $\lambda$, $|\lambda|=1$ with
$$\vp(x)^\ast =\lambda\vp(x)$$
for all $x\in \chi(Z_1)$.
For every non-zero $x\in\chi(Z_1)$, $\{x,x,x\}=\langle x,x\rangle x$, and so, $\vp(\{x,x,x\})=\langle x,x\rangle \vp(x)$.
On the other hand,
\begin{equation*}
\vp(\{x,x,x\})=\{\vp(x),\vp(x),\vp(x)\}
=2\langle \vp(x),\vp(x)\rangle \vp(x)-\langle \vp(x),\vp(x)^\ast \rangle \vp(x)^\ast.
\end{equation*}
The scalar product $\langle \vp(x),\vp(x)^\ast \rangle $ is non-zero, because otherwise we could conclude that $2\vp(x)^2=\langle \vp(x),\vp(x)^\ast \rangle =0$, which would show that $\vp(x)$ is a minimal element by \cite[Proposition 1]{Harris-AnalyinvariantsandtheSchwarzPickinequality}, i.e.\ for every $z\in Z_2$ exists a $\mu\in \C$ such that $\{x,z,x\}=\mu x$). Since $\vp$ is injective, $x$ would be minimal and therefore $xx^\ast =x^2=0$, too. This is a contradiction since $x\neq0$.
Consequently, we can divide by $\langle \vp(x),\vp(x)^\ast \rangle$, and if we put
\[
\lambda:=\frac{2\langle\vp(x),\vp(x)\rangle -\langle x,x\rangle}{\langle \vp(x),\vp(x)^\ast \rangle }
\]
we obtain
\[
\vp(x)^\ast =\lambda\vp(x).
\]
That $\lambda$ does not depend on $x$ can be seen by standard arguments.
Now we choose $\mu$ with $\mu^2=\lambda$ and define $U:=\mu\vp$ . For every $x\in\chi(Z_1)$ we have \[
U(x)=\mu\vp(x)=\mu\overline{\lambda}\vp(x)^\ast =(\mu\vp(x))^\ast =U(x)^\ast
\]
and therefore
$U(z^\ast )=U(z)^\ast $
for every $z\in Z_1$.

Since the rank of the Cartan factor $C^{\textrm{I}}_{1,n}$ is equal to $1$ for
all $n\in\N$, there cannot be any non-zero JB*-triple
morphisms from Cartan factors with rank greater or equal to $2$
to $C^{\textrm{I}}_{1,n}$.

\paragraph{The Hilbertian case} is left to the reader.

\begin{defi}
We will say that a morphism between finite dimensional JC*-triples is of standard form iff it is
one of the morphisms appearing in Theorem~\ref{MorphClass}.
\end{defi}

\begin{cor}\label{K-morph-table}
Let  $Z_1$ and $Z_2$ be Cartan factors of type I, II, or III.
Then for each map $\Phi: K_0(Z_1)\to K_0(Z_2)$ and $\alpha,\beta\in\N_0$ satisfying the
following conditions, there is, up to unitary equivalence, exactly one JB*-morphism $\vp:Z_1\to\Z_2$
with multiplicities $\alpha,\beta$.
\begin{equation*}
\begin{array}{|c|c|c|c|}
  \hline
                 &                      &                  &                  \\
  \longrightarrow& K_0(C^{\textrm{I}}_{M,N})=\Z\oplus\Z & K_0(C^{\textrm{II}}_N)=\Z & K_0(C^{\textrm{III}}_N) =\Z \\
                 &                      &                  &                  \\ \hline
                 &                      &                  &                  \\
  K_0(C^{\textrm{I}}_{m,n}) &(k_1,k_2)\mapsto      & (k_1,k_2)\mapsto & (k_1,k_2)\mapsto \\
  =\Z\oplus\Z    &(\alpha k_1+\beta k_2,
                   \beta k_1+\alpha k_2)&  \alpha(k_1+k_2) &  \alpha(k_1+k_2) \\
                 &\alpha m+\beta n\leq M,
                  \alpha n+\beta m\leq N&\alpha(m+n)\leq N &\alpha(m+n)\leq N \\
                 &                      &                  &                  \\ \hline
                 &                      &                  &                  \\
  K_0(C^{\textrm{II}}_n) =\Z & k\mapsto(\alpha k,\alpha k) &k\mapsto\alpha k  &k\mapsto 2\alpha k\\
                 &\alpha n\leq\min\{M,N\}                  & \alpha n\leq N   & 2\alpha n\leq N  \\
                 &                      &                  &                  \\ \hline
                 &                      &                  &                  \\
  K_0(C^{\textrm{III}}_n) =\Z & k\mapsto(\alpha k,\alpha k)&k\mapsto 2\alpha k&  k\mapsto\alpha k\\
                 &\alpha n\leq\min\{M,N\}                  & 2\alpha n\leq N  & \alpha n\leq N   \\
                 &                      &                  &                  \\ \hline
\end{array}
\end{equation*}
\end{cor}
\begin{proof}
The proof of this result essentially consists in a calculation of $K_0(\vp)$ for the different cases. Note that unitary equivalence induces inner automorphisms on the left C*-algebra of the universal TRO,
so that in order to calculate the action of morphisms between K-groups, we may suppose that all of them are in standard
form.
We illustrate these calculations in the case where $\vp$ maps a rectangular factor $Z$ into a hermitian one, $W$.
In this case, the map
\begin{equation*}
\tau(\vp):\tau(Z)=\M_{m,n}\oplus\M_{n,m}\to\tau(W)=\M_N
\end{equation*}
is
\begin{equation*}
(A,B)\longmapsto
\begin{pmatrix}
    &        &   &   &   &       & A\\
    &        &   &   &   &\iddots &\\
    &        &   &   & A &       &\\
    &        &   & 0 &   &       &\\
    &        &B  &   &   &       &\\
    &\iddots &   &   &   &       &\\
  B &        &   &   &   &       &
\end{pmatrix},
\end{equation*}
with $\alpha$ repetitions of $A$ and $B$, respectively.
For $\cL(\tau(\vp)):\M_m\oplus\M_n\to\M_N$, we have
\begin{multline*}
   \cL(\tau(\vp))(S,T)=\cL(\tau(\vp))\left(\sum S_iS_i^\ast ,\sum TT_i^\ast \right)=\\
   =\left(\sum \tau(\vp)(S_i)\tau(\vp)(S_i)^\ast ,\sum \tau(\vp)(T)\tau(\vp)(T_i)^\ast \right)
   =\diag_\alpha\left(S,\ldots,S\right)\oplus 0\oplus\diag_\alpha\left(T,\ldots,T\right),
\end{multline*}
giving
\begin{equation*}
    K_0(\vp)(k_1,k_2)=\alpha(k_1+k_2),\qquad 2\alpha\leq N
\end{equation*}
Once these calculations are done, the data obtained suffice to unambiguously identify the underlying
JB*-morphism.
\end{proof}

\section{A complete isomorphism invariant}

\paragraph{The canonical involution}
On a TRO $T$, denote the operators $z\mapsto xy^*z$ and $z\mapsto zy^*x$ by $L_{xy^*}$ and $R_{y^*x}$,
respectively. For $T^{\op}$, we use the notation $L^{\op}_{xy^*}$ and $R^{\op}_{y^*z}$. Since $L^{\op}_{xy^*}=R_{y^*x}$
and $L^{\op}_{vw^*}L^{\op}_{xy^*}=R_{w^*v}R_{y^*x}$, as operators on the vector space $T$,
we may identify $\mathcal{L}(T^{\op})$ with $\mathcal{R}(T)$.
If we consider the canonical antiautomorphism $\Phi$ of $\tau(Z)$ as an isomorphism between $\tau(Z)$ and $\tau(Z)^{\op}$
we obtain an isomorphism
\begin{equation*}
  \mathcal{L}(\Phi):\mathcal{L}(\tau(Z))\to\mathcal{L}(\tau(Z)^{\op})=\mathcal{R}(\tau(Z)).
\end{equation*}
At the level of K-theory, we will compare this map to the \emph{Morita isomorphism}: Fix a TRO $T$.
Denote by $\iota_{\mathcal{L}}$ and $\iota_{\mathcal{R}}$ the canonical embeddings of $\cL(T)$ and $\cR(T)$ into
$\Li(T)$, respectively. Then, for a (separable) TRO $T$ the mapping
\begin{equation*}
  \eta_T:=K_0(\iota_{\mathcal{L}})^{-1}\circ K_0(\iota_{\mathcal{R}}):K_0(\mathcal{R}(T))\to K_0(\mathcal{L}(T))
\end{equation*}
is an isomorphism between left and right $K_0$-groups of $T$.
See e.g.\ \cite[Section 2]{BoWe2}.

\begin{defi}
Let $Z$ be a separable JC*-triple, and denote by $\Phi$ the canonical isomorphism $\tau(Z)\to\tau(Z)^{\op}$.
We call the conjugate linear, self-adjoint isomorphism
\begin{equation*}
    \sigma_Z=\eta_{\tau(Z)}K_0\left(\mathcal{L}(\Phi)\right)
\end{equation*}
the \emph{canonical involution} of $K_0(Z)$.
\end{defi}

\begin{lemma}\label{calc canonical involution}
Suppose $Z$ is a JB*-factor. Then the canonical involution is the identity in case
$Z$ is symplectic or hermitian, and, for the rectangular, non-Hilbertian, factor,
\begin{equation*}
    \sigma_Z=\Z\oplus\Z\to\Z\oplus\Z,\qquad (u,v)\mapsto (v,u).
\end{equation*}
\end{lemma}
\begin{proof}
As the $K_0$-groups of the symplectic and hermitian factors are $\Z$, only the case of a factor $Z=\M_{m,n}$
requires a proof. In this case, $\tau(Z)=\M_{m,n}\oplus\M_{n,m}$, $\cL(\tau(Z))=\M_m\oplus\M_n$,
$\cR(\tau(Z))=\M_n^{\op}\oplus\M_m^{\op}$, and
\begin{equation*}
    \Li(Z)=\begin{pmatrix}\M_m\oplus\M_n         & \M_{m,n}\oplus\M_{n,m} \\
                          \M_{n,m}\oplus\M_{m,n} & \M_n^{\op}\oplus\M_m^{\op}
             \end{pmatrix}\cong
           \begin{pmatrix}\M_m     & \M_{m,n} \\
                          \M_{n,m} & \M_n^{\op}
             \end{pmatrix}\oplus
           \begin{pmatrix}\M_n     & \M_{n,m} \\
                          \M_{m,n} & \M_m^{\op}
             \end{pmatrix}.
\end{equation*}
Since the canonical involution $\Phi$ on $\tau(Z)$ is given by $\Phi(A\oplus B)=(B^\top\oplus A^\top)$ we obtain for
\begin{equation*}
\mathcal{L}(\Phi):\M_m\oplus\M_n\to\cL\left(\left(\M_{m,n}\oplus\M_{n,m}\right)^{\op}\right)=\M_n^{\op}\oplus\M_m^{\op}
\end{equation*}
that
\begin{equation*}
\mathcal{L}(\Phi)(A,B)=(B^\top,A^\top)
\end{equation*}
 and $K_0\left(\mathcal{L}(\Phi)\right)(x\oplus y)=y\oplus x$.
Furthermore, the mappings
\begin{equation*}
\iota_{\cL}:M_m\oplus M_n\to
\left(\begin{smallmatrix}\M_m &\M_{m,n}\\ \M_{n,m} &\M_n^{\op}\end{smallmatrix}\right)\oplus
\left(\begin{smallmatrix}\M_n &\M_{n,m}\\ \M_{m,n} &\M_m^{\op}\end{smallmatrix}\right)
\end{equation*}
and
\begin{equation*}
\iota_{\cR}:M_n\oplus M_m\to
\left(\begin{smallmatrix}\M_m &\M_{m,n}\\ \M_{n,m} &\M_n^{\op}\end{smallmatrix}\right)\oplus
\left(\begin{smallmatrix}\M_n &\M_{n,m}\\ \M_{m,n} &\M_m^{\op}\end{smallmatrix}\right)
\end{equation*}
 are given by
\begin{equation*}
    \iota_{\cL}(p\oplus q)=\begin{pmatrix}p & 0\\0 & 0\end{pmatrix}\oplus \begin{pmatrix}q & 0\\0 & 0\end{pmatrix}
    \qquad\text{as well as}\qquad
    \iota_{\cR}(r\oplus s)=\begin{pmatrix}0 & 0\\0 & r\end{pmatrix}\oplus \begin{pmatrix}0 & 0\\0 & s\end{pmatrix},
\end{equation*}
and we find $\eta_Z=\id_{\Z^2}$.
\end{proof}

It is also not difficult to calculate the canonical involution in the Hilbertian and spin case, respectively,
but we won't need that here.

\paragraph{Special subsets of the $K_0$-group. The invariant.}
In order to obtain an isomorphism invariant for finite dimensional JB*-triples, K-groups have to be
furnished with additional structure. We will use here a variant which accomplishes the classification of
the most complicated part of the inductive limit of finite dimensional JB*-triples. Besides the canonical
involution, this invariant consists of three distinguished subsets of $K_0(Z)$, defined in the following way.

\begin{defi}
Denote by $\Sigma(Z)$ the subset of $K_0(Z)$ consisting of all equivalence classes coming from projections
in $\cL(\tau(Z))$.
This subset will be called \emph{the scale of $Z$}.
\end{defi}

 A subset of $\Sigma(Z)$ is obtained as follows.

\begin{defi}
Suppose that $Z$ is embedded into its universal TRO, and fix a tripotent $u\in Z$.
Since $u$ also is tripotent with respect to the
TRO-product, the element $\rho_Z(u)\rho_Z(u)^\ast$ is a projection in the left
C*-algebra of $TRO(Z)$, and we denote the set of the equivalence classes in $K_0(Z)$
which arise in this way by $\Delta^+(Z)$.
\end{defi}

As $\rho_Z(u)^\ast\rho_Z(u)$
is equivalent in $\Li(\tau(Z))$ to $\rho_Z(u)\rho_Z(u)^\ast$, no new information is gained from
using right C*-algebras.

\begin{defi}
Let $Z$ be a
JC*-triple with canonical antiautomorphism $\Phi:\tau(Z)\to\tau(Z)$.
We
denote by $Z^-$ the JC*-triple
\begin{equation*}
    Z^-=\set{z\in\tau(Z)}{\Phi(z)=-z},
\end{equation*}
and let $\Delta^-(Z)=\Delta^+(Z^-)$.
\end{defi}

We are ready for

\begin{defi}
The $K^{\pm}$-invariant of a JB*-triple $Z$ is the quintuple
\begin{equation*}
    K^\pm_0(Z)=(K_0(Z),\Sigma(Z),\Delta^+(Z),\Delta^-(Z),\sigma_Z)
\end{equation*}
\end{defi}

\begin{prop}
For the finite dimensional universally reversible factors,
\begin{gather*}
    K^\pm_0(C_n^{\textrm{III}}) = \left(\Z,\{0,1,\ldots,n\},\{0,2,4,\ldots,k\},\{0,1,2,\ldots,n\},\id\right)\\
    K^\pm_0(C_n^{\textrm{II}}) = \left(\Z,\{0,1,\ldots,n\},\{0,1,2,\ldots,n\},\{0,2,4,\ldots,k\},\id\right)
\end{gather*}
and,
with $k$ being the largest even number such that $k\leq n$ and $\mu=\min\{m,n\}$,
\begin{multline*}
    K^\pm_0(C_{m,n}^{\textrm{I}}) =\left(\Z^2,\{0,1,\ldots,m\}\times\{0,1,\ldots,n\},\right.\\
      \left.\{(0,0),(1,1),\ldots,(\mu,\mu)\},\{(0,0),(1,1),\ldots,(\mu,\mu)\},\sigma:(s,t)\mapsto (t,s)\right).
\end{multline*}
\end{prop}
\begin{proof}
This follows easily from the observation that $Z^-$ is symplectic (hermitian) if $Z$ is hermitian
(symplectic), that for a rectangular factor $W$,
\begin{equation*}
    W^-=\set{(A,-A^\top)\in\M_{m,n}\oplus\M_{n,m}}{A\in W}\cong W,
\end{equation*}
and the calculation of the $\Delta$-sets carried out in \cite[Proposition 4.4]{BoWe3}. The remaining
part follows from Lemma~\ref{calc canonical involution}.
\end{proof}

%

\paragraph{Functoriality of the invariant}
One of the most important results in this section is the fact that Morita equivalence
is well behaved under TRO-morphisms.

\begin{lemma}\label{Functoriality Morita}
Whenever $\vp: T\to U$ is a TRO-morphism of separable TROs, then the diagram
\begin{equation*}
\text{\xymatrix{ K_0(\mathcal{R}(T))\ar[d]_{K_0(\mathcal{R}(\vp))}\ar[rr]^{\eta_T}
 &&K_0(T) \ar[d]^{K_0(\vp)}
\\
K_0(\mathcal{R}(U))\ar[rr]_{\eta_U}
 && K_0(U)}}
\end{equation*}
commutes.
\end{lemma}

Again, see \cite[Section 2]{BoWe2} for more details.

\begin{lemma}\label{real form properties}
Let $\phi:X\to Y$ be a
morphism between
JC*-triples, and denote by $\Phi_X$ and $\Phi_Y$ their canonical antiautomorphisms.
Then
\begin{equation*}
    \tau(\phi)\circ\Phi_X=\Phi_Y\circ\tau(\phi).
\end{equation*}
\end{lemma}
\begin{proof}
Since $X$
generates $\tau(X)$ as TRO, $\Phi_X$ (and similarly $\Phi_Y$) is uniquely determined by the relation
\begin{equation*}
    \Phi_X\left(x_1\prod_{i=1}^{n}x_{2i}^\ast x_{2i+1}\right)=x_{2n+1}\prod_{i=1}^{n}x_{2n-2i+2}^\ast x_{2n-2i+1},
    \qquad x_1,\ldots,x_{2n+1}\in X.
\end{equation*}
Thus, for all $x_1,\ldots,x_{2n+1}\in X$,
\begin{multline*}
    \tau(\phi)\left(\Phi_X\left(x_1\prod_{i=1}^{n}x_{2i}^\ast x_{2i+1}\right)\right)=\\
    =\tau(\phi)\left(x_{2n+1}\right)\prod_{i=1}^{n}\tau(\phi)\left(x_{2n-2i+2}\right)^\ast\tau(\phi)\left(x_{2n-2i+1}\right)
    =\Phi_Y\left(\tau(\phi)\left(x_1\prod_{i=1}^{n}x_{2i}^\ast x_{2i+1}\right)\right),
\end{multline*}
proving the lemma.
\end{proof}

Lemma~\ref{real form properties} implies

\begin{lemma}\label{delta properties}
Let $Z$ be a
If $\mu:Z_1\to Z_2$ is
a JC*-morphism,
then
\begin{equation*}
    \mu^-=\rest{\tau(\mu)}{Z_1^-}
\end{equation*}
is a well-defined JC*-morphism $Z_1^-\to Z_2^-$.
\end{lemma}

\begin{prop}\label{functoriality lemma}
The $K_0^\pm$-invariant is functorial and additive, i.e.\ for any JB*-morphism $\vp$,
$K^\pm_0(\vp)$ maps scales into scales, $\Delta^+$-- and $\Delta^-$--sets into $\Delta^+$-- and $\Delta^-$--sets,
intertwines canonical involutions, and, if $Z=Z_1\oplus\cdots\oplus Z_k$ then
\begin{equation*}
    K_0^\pm(Z)=\bigoplus_{\kappa=1}^k
    \left(K_0(Z_\kappa),\Sigma(Z_\kappa),\Delta^+(Z_\kappa),\Delta^-(Z_\kappa),\sigma_{Z_\kappa}\right).
\end{equation*}
\end{prop}
\begin{proof}
Functoriality of K-groups and scales is obvious. The statement for $\Delta$-sets and the involution follows
easily from Lemma~\ref{delta properties} and Lemma~\ref{real form properties}.
Additivity of the K-functor was shown in \cite[Theorem 3.2.]{BoWe1}.
$\Delta^+$ clearly behaves additively. That the canonical involution behaves additively follows from the
fact that $\tau$ is additive
and the uniqueness of the canonical antiautomorphism \cite[Theorem 3.5]{BunceFeelyTimoneyI}.
This fact established, additivity of $\Delta^-$ and $\sigma$ follow as well.
\end{proof}
\begin{defi}
We call a finite dimensional JB*-triple \emph{principal} iff it is the direct sum of non-Hilbertian factors
of type I, II, or III.
\end{defi}
This definition owes its existence to the statement of the main theorem, less so to the fact that in many of the
other cases analogous results would be impossible to prove.
\begin{thm}\label{JB*-complete invariant}
Let $Z_1$ and $Z_2$ be finite-dimensional principal $JB^\ast$-triples.
If $\gamma:K_0(Z_1)\to K_0(Z_2)$ is a morphism of groups
which
\begin{description}
  \item[(i)]
  maps the subsets in $K^\pm_0(Z_1)$ into the corresponding subsets of $K^\pm_0(Z_2)$,
  and
  \item[(ii)]
  intertwines canonical involutions,
\end{description}
then, up to unitary equivalence, there is exactly one morphism $\Phi:Z_1\to Z_2$ such that $\gamma=K_0(\Phi)$.
\end{thm}
\begin{proof}
Suppose first that both, $Z_1$ and $Z_2$ are principal factors. In all the cases, K$_0$-groups are either
$\Z^2$ or $\Z$, and since (ii) was supposed to hold, $\gamma$ must be represented by
\begin{equation*}
    \begin{pmatrix}
    \alpha & \beta\\
    \beta  & \alpha
    \end{pmatrix},\qquad
    \begin{pmatrix}
    \alpha\\
    \alpha
    \end{pmatrix},\qquad
    \begin{pmatrix}
    \alpha & \alpha
    \end{pmatrix},
    \qquad\alpha,\beta\in\Z,
\end{equation*}
in case $C^{\textrm{I}}\to C^{\textrm{I}}$, $C^{\textrm{I}}\to C^{\textrm{II,III}}$, and
$C^{\textrm{II,III}}\to C^{\textrm{I}}$. In all the other cases the mulitplicity matrix
is of the more simple form $(\alpha)$ where $\alpha\in\Z$.
In all cases except for a morphism $C_{m,n}^I\to C_{M,N}^I$, the condition $\gamma(\Delta^\pm(Z_1))\sub
\Delta^\pm(Z_2)$ ensues $\alpha\geq 0$ and, furthermore, forces $\alpha$ to respect the upper bounds
from Corollary~\ref{K-morph-table}.
The same condition guarantees
that the multiplicities of the morphisms between different
factors of Type II and III are even.
For a morphism $C_{m,n}^I\to C_{M,N}^I$, one uses the fact that
\begin{multline*}
\gamma(\Sigma(Z_1))=\set{(\alpha s+\beta t,\beta s+\alpha t)}{s=0,1,\ldots,m,\ t=0,1,\ldots,n}\sub\\
\sub\Sigma(Z_2)=\set{(S,T)}{S=0,1,\ldots,M,\ T=0,1,\ldots,N}
\end{multline*}
to see that $\alpha,\beta\geq 0$ with $\alpha m+\beta n\leq M$ and
$\alpha n+\beta m\leq N$.
In all cases, it then follows from Corollary~\ref{K-morph-table}
that there exists a uniquely determined morphism $\Phi:Z_1\to Z_2$ with $\gamma=K_0(\Phi)$.

For the case of arbitrary finite dimensional principal JB*-triples, let, for $i=1,2$,
$Z_i=\bigoplus_{\nu=1}^{N_i}Z_{i,\nu}$ be the decomposition
of $Z_i$ into factors. Write $\gamma=\left(\gamma_{\mu\nu}\right)$, where
$\gamma_{\mu\nu}:K_0(Z_{1,\mu})\to K_0(Z_{2,\nu})$ are morphisms of abelian groups.
Because the $K^\pm$-invariant
is additive, each of the $\gamma_{\mu\nu}$ is a $K^\pm$-morphism so that by the first part of this
proof, there are
morphisms $\Phi_{\mu\nu}:Z_{1,\mu}\to Z_{2,\nu}$ with $K_0(\Phi_{\mu\nu})=\gamma_{\mu\nu}$. Thus,
$\Phi=\left(\Phi_{\mu\nu}\right)$ is the morphism we were looking for.
\end{proof}

\paragraph{Inductive limits}
In the following, $Z$ will denote an inductive limit of JC*-triple systems $Z_n$ with morphisms $\mu_n:Z_n\to Z$, and we
think of $Z$ as $Z=\overline{\bigcup_{n\in\N}\mu_n(Z_n)}$.

We will eventually show that $K^\pm$ is continuous under the formation of inductive limits. K-theory itself
possesses this continuity property for JB*-triples, as was shown in \cite{BoWe2}, so that
\begin{equation*}
    K_0(Z)=\bigcup_{n\in\N}K_0(\mu_n)(Z_n).
\end{equation*}
That $\Sigma(\cdot)$ is continuous with respect to
inductive limits follows from combining well-known facts in the following way. Extending the connecting morphisms $\vp_n:Z_n\to Z_{n+1}$ to TRO-morphism $\tau(\vp_n):\tau(Z_n)\to\tau(Z_{n+1})$, we
can form the corresponding inductive limit in the TRO-category. As shown in \cite[Lemma~3.5]{BoWe2}, all of
the functors $\tau$, $\cL$, and $\cR$ are continuous, and so continuity of $\Sigma$ follows from \cite[IV.1.2, IV.1.3]{Davidson-Calgebrasbyexample}.

Central to the treatment of tripotents in inductive limits of JC*-triples is
\begin{lemma}\label{aprox of tripotents}
Let $Z$ be a JB*-triple system and $z\in Z$ with $\|z\|<1$. If $\|\{z,z,z\}-z\|<\varepsilon\leq\frac{1}{4}$, then there exists a tripotent $e\in Z$ with $\|e-z\|<\varepsilon^{1/2}$.
\begin{proof}
It is well known \cite[Lemma 1.14]{Kaup-ARiemannMappingTheoremForBoundedSymmetricDomainsInComplexBanachSpaces}
that $\fT$, the JB*-triple generated by $z$, is, through a Gelfand transform $x\mapsto\hat{x}$,
triple (and thus isometrically) isomorphic to
a commutative $C^\ast $-algebra $\fC_0(S,\C)$, equipped with its canonical triple structure structure.
The assumption yields, for all $s\in S$,
\begin{equation*}
   \left| 1-|\hat{z}(s)|\right\|\hat{z}(s)|<\frac{\eps}{1+|\hat{z}(s)|}\leq\eps
\end{equation*}
and so for the spectrum of $z$ within $\fT$ we find $\sigma_\fT(z)=\hat{z}(S)\sub U_0\cup U_1$, where
\begin{equation*}
    U_0=\set{c\in\C}{|c|<\eps^{1/2}\leq\frac{1}{2}}\quad\text{and}\quad
    U_1=\set{c\in\C}{|1-|c\|<\eps^{1/2}\leq\frac{1}{2}}.
\end{equation*}
Note that the sets $\hat{z}^{-1}(U_{0,1})$ are open and closed, and $\hat{z}^{-1}(U_1)$ is compact. Let
\begin{equation*}
\hat{e}(s)=
    \begin{cases}
    0                               &\text{if $\hat{z}(s)\in U_0$}\\
    \hat{z}(s)|\hat{z}(s)|^{-1}     &\text{if $\hat{z}(s)\in U_1$}.
    \end{cases}
\end{equation*}
Then $\hat{e}\in \fC_0(S,\C)$ and
$
\|\hat{z}-\hat{e}\|<
\eps^{1/2}
$,
whence $e\in\fT$ is the tripotent we were looking for.
\end{proof}
\end{lemma}

\begin{cor}
Let $Z$ be the inductive limit of a directed system of $JC^\ast $-triple systems $(Z_n)$, with canonical morphisms
$\mu_{n}:Z_n\to Z$.
If $e\in Z$ is tripotent and $1\geq\varepsilon>0$, then there exist $n\in\N$ and a tripotent $f\in Z_n$
such that
\[
\|e-\mu_n(f)\|<\varepsilon.
\]
\begin{proof}
Since $\mu_n$ and $\{\cdot,\cdot,\cdot\}$ are continuous, we may pick $n\in\N$ as well as $e_n\in Z_n$ with $\|e-\mu_n(e_n)\|<\eps/2$ as well as $\|\{e_n,e_n,e_n\}-e_n\|<\eps^2/4$. The previous Lemma
produces a tripotent $f\in Z_n$ with $\|e_n-f\|<\eps/2$, and we are done.
\end{proof}
\end{cor}
This last lemma shows in light of \cite[IV.1.2]{Davidson-Calgebrasbyexample} that
$\Delta^+(Z)=\bigcup_{n\in\N}K_0(\mu_{n})\left(\Delta^+(Z_n)\right)$.
In order to treat $\Delta^-(Z)$ and the canonical automorphism of $Z$ in a similar way,
we prove the following

\begin{lemma}\label{Limits of automorphisms}
Let $Z$ be the inductive limit of finite dimensional universally reversible JC*-triple systems.
Then $Z$ is universally reversible. Furthermore,
\begin{description}
  \item[(a)]
  the involutive antiautomorphisms $\Phi$ and $\Phi_n$ of $Z$ and $Z_n$ satisfy
    \begin{equation*}
    \rest{\Phi}{\tau\left(\mu_n\right)\left(\tau\left(Z_n\right)\right)}=\tau\left(\mu_n\right)\Phi_n,
    \end{equation*}
  \item[(b)]
   the canonical involutions $\sigma$ and $\sigma_n$ of $K_0(Z)$ and $K_0(Z_n)$, respectively, satisfy
    \begin{equation*}
    \rest{\sigma}{K_0(\mu_n)(K_0(Z_n))}=K_0(\mu_n)\sigma_n.
    \end{equation*}
\end{description}
\begin{proof}
Part (a)is a consequence of Lemma~\ref{real form properties},
whereas (b) follows from Lemma~\ref{Functoriality Morita}. In order to see that $Z$ is universally reversible,
we use the fact that, with the canonical antiautomorphism $\Phi$ of $Z$ given in (a), $z\in\tau(Z)$ is
in $Z$ as soon as $\Phi(z)=z$ as well as \cite[Lemma 4.2]{BunceFeelyTimoneyI}.
\end{proof}
\end{lemma}
From Lemma~\ref{Limits of automorphisms}(a) we may now conclude that $Z^-$ is the inductive limit of the system
$(Z_n^-,\vp_n^-)$ with canonical mappings $\mu_n^-:Z_n\to Z$. As a consequence, the same argument showing that
$\Delta^+(Z)$ is the union of sets $\mu_n(\Delta^+(Z_n))$ now shows that the analogous statement holds for $\Delta^-(Z)$.
We summarize this discussion in

\begin{prop}\label{Kpm and limit}
Let $\left(Z,(\mu_n)\right)$ be the inductive limit of finite dimensional universally reversible JC*-triples systems $(Z_n)$.
Then
\begin{align*}
K_0(Z)&=\bigcup_{n\in\N}K_0(Z_n)\\
\Sigma(Z)&=\bigcup_{n\in\N}K_0(\mu_{n})\left(\Sigma(Z_n)\right)\\
\Delta^+(Z)&=\bigcup_{n\in\N}K_0(\mu_{n})\left(\Delta^+(Z_n)\right)\\
\Delta^-(Z)&=\bigcup_{n\in\N}K_0(\mu_{n})\left(\Delta^-(Z_n)\right)\quad\text{and}\\
\sigma_Z &=\lim_{\longrightarrow} \sigma_{Z_n}.
\end{align*}
\end{prop}

\section{Main Results}
This section is devoted to the proof of
\begin{thm}\label{main}
Each inductive limit of finite dimensional JB*-triples $Z$ admits a decomposition
\begin{equation*}
    Z=P\oplus E\oplus S\oplus H
\end{equation*}
where $H$ is a direct sum of Hilbert spaces, $S$ a direct sum of spin factors and $E$
a direct sum of exceptional factors.

The principal part $P$ of $Z$ is an inductive limit
of hermitian, symplectic and rectangular factors and can be classified by its K$^\pm$-invariant
which is, in a suitable sense, the inductive limit of the involved finite dimensional JB*-triples.
\end{thm}

\begin{defi}
Let $Z$ be a JB*-triple.
\begin{description}
\item[(a)]
We call $Z$ an \emph{AF-triple} if it is the inductive limit of finite dimensional JB*-triple systems.
\item[(b)]
$Z$ is called a \emph{principal AF-triple} if it is an inductive limit
of (sums of) hermitian, symplectic and rectangular factors.
\end{description}
\end{defi}

\begin{prop}
If $\fA$ is a $C^\ast $-algebra which is the inductive limit of finite dimensional JB*-triple systems, then it is already an AF-algebra.
\end{prop}
This can be seen in the following way: Let $(Z_n)$ be a sequence of finite dimensional JB*-triples such that $\fA=\overline{\bigcup Z_n}$. If we define $\fA_n$ to be the $C^\ast $-algebra generated by $Z_n$ in $\fA$ for $n\in\N$, then $\fA_n$ is finite dimensional since the inclusion mapping of $Z_n$ into $\fA$ induces a surjective $*$-morphism from the universal enveloping $C^\ast $-algebra of $Z_n$ onto $\fA_n$.

\begin{lemma}\label{lemma below}
Each principal inductive limit $Z=\overline{\bigcup_{n=1}^\infty Z_n}$ is isomorphic to an inductive limit for which the connecting maps are in standard form.
\begin{proof}
We denote by $\vp_n:Z_n\to Z_{n+1}$ the morphisms given at the outset, and by $\sigma_n:W_n\to W_{n+1}$ the standard forms for
these mappings, for all $n\in \N$. Pick unitaries $U_n\in\cL\tau(Z_n)$ and $\widehat{U}_n\in\cR\tau(Z_n)$ so that
$\vp_n=U_{n+1}\sigma_n\widehat{U}_{n+1}$. Put $V_1=\widehat{V}_1=1$, $\psi_1=\id_{Z_1}$, and define, recursively,
a sequence of unitaries in $\cL\tau(Z_n)$ and $\cR\tau(Z_n)$ by letting
\begin{equation*}\
V_{n+1}=\left(\cL\vp_n(V_n)+1-\cL\vp_n(1)\right)U_{n+1}\qquad
\widehat{V}_{n+1}=\widehat{U}_{n+1}\left(\cR\vp_n(\widehat{V}_n)+1-\cR\vp_n(1)\right).
\end{equation*}
Define JB*-automorphisms $\psi_n$ of $Z_n$ through
$\psi_n(z)=V_n z\widehat{V}_{n}$.
Then, for all $z\in Z_n$,
\begin{multline*}
\vp_n\psi_n(z)=
\left(\cL\vp_n(V_n)+1-\cL\vp_n(1)\right)\vp_n(z)\left(\cR\vp_n(\widehat{V}_n)+1-\cR\vp_n(1)\right)=\\
=V_{n+1}\sigma_n(z)\widehat{V}_{n+1}=\psi_{n+1}\sigma_n(z),
\end{multline*}
giving the result.
\end{proof}
\end{lemma}

\begin{thm}\label{AF-Triple klassifizierung}
If $Z$ and $W$ are principal AF-triples and $\gamma:K_0(Z)\rightarrow K_0^{\text{\tiny JB*}}(W)$ is an isomorphism respecting the $K^\pm_0$-invariant, then there exists a JB*-isomorphism $\vp:Z\to W$ with $K_0^{\text{\tiny JB*}}(\varphi)=\sigma$.
\begin{proof}
Suppose that $Z$ and $W$ as inductive limits are given by $\zeta_n:Z_{n-1}\to Z_n$ and $\eta_n:W_{n-1}\to W_n$,
respectively.
By Lemma~\ref{lemma below}, we may suppose that the connecting morphisms are in standard form.
By Propopsition~\ref{Kpm and limit}, $K_0(Z)=\bigcup_{n\in\N_0}K_0(Z_n)$
and $K_0(W)=\bigcup_{n\in\N_0}K_0(W_n)$.
Consider the restriction $\gamma_1$ of $\gamma$ to $K_0(Z_0)$. Since $K_0(Z_0)$ is a finitely generated $\Z$-module
(and $\gamma$ a $\Z$-module map),
its image under $\gamma_1$ is one as well, and the image of $\gamma_1$
must be contained in some $K_0(W_{n_1})$, for $n_1\in\N$ large enough.
Invoking Proposition~\ref{Kpm and limit} as well as the additivity of the $K^\pm$-invariant,
\begin{equation*}
    \gamma_1\left(\Delta^\pm(Z_0)\right)\sub \bigcup_{n\in\N_0}\Delta^\pm\left(W_n\right)\qquad\text{and}\qquad
    \gamma_1\left(\Sigma(Z_0)\right)\sub \bigcup_{n\in\N_0}\Sigma\left(W_n\right).
\end{equation*}
Because each of the sets $\Delta^\pm(Z_0),\Sigma(Z_0)$ is finite, we may enlarge $n_1$, if necessary, and
assume that
\begin{equation*}
    \gamma_1\left(\Delta^\pm(Z_0)\right)\sub \Delta^\pm\left(W_{n_1}\right)\qquad\text{and}\qquad
    \gamma_1\left(\Sigma(Z_0)\right)\sub \Sigma\left(W_{n_1}\right).
\end{equation*}
As $\gamma$ intertwines the
canonical involutions on the involved K-groups, Propopsition~\ref{Kpm and limit} shows that $\gamma_1$ intertwines
the canonical involutions of $K_0(Z_0)$ and $K_0(W_{n_1})$, respectively.
Consequently, $\gamma_1$ is a morphism of commutative groups of the type that was addressed in Theorem~\ref{JB*-complete invariant},
and it follows that there is a triple morphism $\phi_1:Z_0\to W_{n_1}$ with $K_0(\phi_1)=\gamma_1$. Proceeding
inductively, we find a strictly increasing sequence $(n_k)$ of integers with $n_0=0$, as well as triple morphisms
$\phi_k:Z_{n_{k-1}}\to W_{n_k}$. Let $\gamma_k=K_0(\phi_k)$. Then
\begin{equation*}
\text{
\begin{xy}
  \xymatrix{
      \ldots \ar[r] & K_0(Z_{n_{k-1}}) \ar[r] \ar[rd]_{\gamma_k}&\ldots\ar[r]   & K_0(Z_{n_k})\ar[r]\ar[dr]_{\gamma_{k+1}}
                                           &\ldots\ar[r]       &\ldots\ar[r]  & K_0(Z) \ar[d]^{\gamma} \\
                    & \ldots \ar[r]        & K_0(W_{n_k})\ar[r] &\ldots\ar[r]
                                           & K_0(W_{n_{k+1}})\ar[r] & \ldots\ar[r] & K_0(W)
  }
\end{xy}
}
\end{equation*}
commutes. Let
$
    \widetilde{\zeta}_{k+1}=\prod_{\nu=n_k+1}^{n_{k+1}}\zeta_\nu
$
and
$
    \widetilde{\eta}_{k+1}=\prod_{\nu=n_k+1}^{n_{k+1}}\eta_\nu
$ so that
\begin{equation*}
K_0(\widetilde{\eta}_{k+1}\phi_k)=K_0(\phi_{k+1}\widetilde{\zeta}_k)
\end{equation*}
We then determine inductively unitaries $U_k$ and $V_k$ with the property that
if we let $\widetilde{\phi}_k= U_k\phi_k V_k$ we have
\begin{equation*}
   \widetilde{\eta}_{k+1}\widetilde{\phi}_k=\widetilde{\phi}_{k+1}\widetilde{\zeta}_k.
\end{equation*}
It follows that there is a morphism $\phi$ sucht that
\begin{equation*}
\text{
\begin{xy}
  \xymatrix{
      \ldots \ar[r] & Z_{n_{k-1}} \ar[r] \ar[rd]_{\widetilde{\phi}_k}&\ldots\ar[r]   & Z_{n_k}\ar[r]\ar[dr]_{\widetilde{\phi}_{k+1}}
                                           &\ldots\ar[r]       &\ldots\ar[r]  & Z \ar[d]^{\widetilde{\phi}}\\
                    & \ldots \ar[r]        & W_{n_k}\ar[r] &\ldots\ar[r]
                                           & W_{n_{k+1}}\ar[r] & \ldots\ar[r] & W
  }
\end{xy}
}
\end{equation*}
commutes, implying $K_0(\widetilde{\phi})=\gamma$.
\end{proof}
\end{thm}

\begin{prop}\label{SH-case}
Let $Z_\infty$ be the inductive limit of a sequence of finite dimensional JB*-triple systems $(Z_n,\vp_n)$.
\begin{description}
\item[(a)]
If each $Z_n$ is a spin factor, then $Z_\infty$ is a spin factor.
\item[(b)]
If all $Z_n$ are finite direct sums of finite dimensional Hilbertian factors, then $Z_\infty$ is a
direct sum of Hilbertian factors.
\item[(c)]
If all $Z_n$ are finite direct sums of finite dimensional spin factors, then $Z_\infty$ is a direct sum of spin factors.
\end{description}
\begin{proof}
We start with (a).
According to Theorem\ref{MorphClass}(iv) we may write $\vp_n(z)=\mu_n U_n(z)$ with $\mu_n$ complex, of modulus one, and $U_n$ a
unitary embedding which respects involutions. It is straightforward to check that the inductive system $(Z_n,U_n)$ has a spin factor $Z_{\infty,0}$ as limit. Let $\nu_1=\id_{Z_1}$ and, for $k\geq 2$,
\begin{equation*}
    \nu_k: Z_k\to Z_k,\quad z\mapsto \overline{\mu}_1\cdots\overline{\mu}_{k-1}z.
\end{equation*}
Then
\begin{equation*}
\begin{xy}
  \xymatrix{
      Z_{1} \ar[r]^{\vp_1} \ar[d]_{\id_{Z_1}}  & Z_2\ar[r]^{\vp_2}\ar[d]_{\nu_2} & Z_3\ar[r]^{\vp_3}\ar[d]_{\nu_3}&\ldots\ar[r]&Z_\infty   \\
        Z_{1} \ar[r]_{U_1} &  Z_2 \ar[r]_{U_2}
        & Z_3\ar[r]_{U_3} & \ldots\ar[r] & Z_{\infty,0}
  }
\end{xy}
\end{equation*}
commutes, and $Z_\infty$ is isomorphic to the spin factor $Z_{\infty,0}$.

The proofs of (b) and (c) are fairly standard arguments. We indicate how to obtain (c).
Write $Z_n$ as a direct sum of spin factors, $Z_n=S_1^n\oplus\ldots\oplus S^n_{k_n}$. If we choose from every level $n\in\N$ exactly one index $\alpha_n$ with $1\leq\alpha_n\leq k_n$, we get a path $I:=(\alpha_i)_{i\in\N}$.
For each path $I=(\alpha_i)$ denote by $(S_{\alpha_i}^I)$ the sequence of corresponding spin factors and write
$\vp^I_{\alpha_i}$ for the result of restricting and compressing $\vp_i$ to a JB*-morphism
$S_{\alpha_i}^I\to S_{\alpha_{i+1}}^I$. We thus obtain an inductive system $(S_{\alpha_i}^I,\vp_{\alpha_i}^I)$ whose
limit we denote by $S^I$. Unless $S^I=0$, all maps $\vp_{\alpha_i}^I$ are injective, and thus, by (a), $S^I$ is a spin factor.
Identify paths if they eventually coincide and denote the resulting set of equivalence classes $[I]$ by $\mathcal{I}$.
We may define $S^{[I]}=S^{I}$, and then find
\begin{equation*}
Z_\infty=\bigoplus_{[I]\in\mathcal{I}}S^{[I]}
\end{equation*}
\end{proof}
\end{prop}

\paragraph{Proof of Theorem~\ref{main}}
We call a class of factors (which we denote here by P, E, S, or H, respectively) \emph{cul-de-sac} if it is not possible
to find a non-trivial JB*-morphism from an object of another class to an object of this class or if it is not possible
to find a non-trivial JB*-morphism from an object of this class to an object of another class. If a class is cul-de-sac,
then it does not interfere with the other classes in the inductive limit. Its fairly obvious that $E$ is cul-de-sac.
Moreover $H$ is cul-de-sac, because all Hilbertian triple systems are of $\rank$ 1 and it is not possible to embed
members of other classes into them. The spin factor, which are the objects from $S$ are all of $\rank$ 2. Thus it is
only possible to embed elements of $H$ and type I factors whose rank equals 2. In both cases it is not possible to go
back to theses classes due to the rank results in section 1. For the same reason, it is impossible to enter another class
starting from $P$, and so this is a cul-de-sac, too.

The result now follows from combining Theorem~\ref{AF-Triple klassifizierung} with Proposition~\ref{SH-case}.

\section{Principal AF-triples}

\begin{cor}\label{principal cor}
An AF-triple $Z$ is principal iff it is universally reversible.
\begin{proof}
Each direct summand of a
universally reversible JC*-triple inherits this property and so, a universally reversible AF-triple must be principal, as a consequence of \cite[Theorem 5.6]{BunceFeelyTimoneyI} and Theorem~\ref{main}.
Since all Cartan factors appearing in a principal AF-triple are universally reversible, the converse follows easily.
\end{proof}
\end{cor}

\begin{defi}
Suppose $\left((Z_\lambda),(\varphi_\lambda)\right)$ is an inductive sequence
of finite di\-men\-sional JB*-triple systems and write $(\alpha^\lambda_{ij})$ for the
multiplicity matrix of $\vp_\lambda$. The \emph{ter\-na\-ry Bratteli diagram}\index{ternary Bratteli diagram}
of $\left((Z_\lambda),(\varphi_\lambda)\right)$ is a weighted, colored graph with the following properties.
\begin{description}
\item[(a)]
Vertices are arranged into rows $V_\lambda$, one vertex for each single factor going into $Z_\lambda$.
\item[(b)]
Each vertex $v$ is colored by a pair $(h,n)$. Here, $h$ is taken from six different hues, \ding{182},$\ldots,$\ding{187},
representing the Cartan factors, whereas $n$ is the dimension of the Cartan factor represented by $v$.
\item[(c)]
The subgraph consisting of vertices in $V_\lambda\cup V_{\lambda+1}$ and the edges between them is obtained from letting $(\alpha^\lambda_{ij})$ be its adjacency matrix, i.e.\ the $i$th vertex in $V_\lambda$ and the $j$th vertex in $V_{\lambda+1}$ are connected by an
edge of weight $\alpha^\lambda_{ij}$ (and no edge should $\alpha^\lambda_{ij}=0$).
\end{description}
\end{defi}


\begin{thm}
Let $(Z_n)$ and $(W_n)$ be increasing sequences of finite dimensional principal $JC^\ast $-triples.
If $Z=\overline{\bigcup_{n=1}^{\infty}Z_n}$ and $W=\overline{\bigcup_{n=1}^{\infty}W_n}$ have the same ternary Bratteli diagram, then they are isomorphic.
\begin{proof}
As Bratteli Diagrams encode precisely the standard form of the connecting maps, the result follows from
Lemma~\ref{lemma below} below.
\end{proof}
\end{thm}

\begin{prop}
Let $Z$ be a  JB*-triple system. Then the following are equivalent:
\begin{description}
\item[(a)] $Z$ is a principal AF-triple.
\item[(b)] $Z$ is separable and for every $\varepsilon >0$ and for all $z_1,\ldots,z_n\in Z$ there exists a finite dimensional principal subtriple $W\sub Z$ such that $\dist(z_i, W)<\varepsilon$ for $i\in\{1,\ldots,n\}$.
\end{description}
\begin{proof}
Clearly, for a principal AF-triple $Z$, (b) is true. Suppose, conversely, that (b) holds.
As (the canonical image of) $Z$ generates $\Li(\tau(Z))$
as a C*-algebra, it follows
that $\Li(\tau(Z))$ satisfies (b), with `principal subtriples' replaced by `sub-C*-algebras'. Hence, according to \cite{Bratteli-InductivelimitsoffinitedimensionalCalgebras}, $\Li(\tau(Z))$ is an AF-algebra, and there is an increasing
sequence $(A_n)$ of finite dimensional sub C*-algebras of $\Li(\tau(Z))$ with $\overline{\bigcup_n A_n}=\Li(\tau(Z))$.
Let
\begin{equation*}
    p=\begin{pmatrix}1 & 0\\0 & 0\end{pmatrix}\in
    \Li(\tau(Z))=\begin{pmatrix}\cR(\tau(Z))& \tau(Z)\\
    \tau(Z)^\ast  &\cL(\tau(Z))\end{pmatrix},
\end{equation*}
and put $P(a)=pa(1-p)$. Then $(T_n)=(P(A_n))$ is an increasing sequence of sub-TROs of $\tau(Z)$.
If, for $z\in Z$, $(t_n)$ is a sequence with $t_n\in T_n$ and $t_n\to z$, then $z_n=P(t_n)$
is in $Z_n$, the sequence $(z_n)$ converges to $z$, and so $\overline{\bigcup_n T_n}=\tau(Z)$.
We now use the fact that $Z$ is principal. By Corollary~\ref{principal cor} and \cite[Lemma 4.2]{BunceFeelyTimoneyI},
the projection $\Pi=1/2(\id+\Phi)$, where $\Phi:\tau(Z)\to \tau(Z)$ is the canoncial antiautomorphism,
maps sub-TROs of $\tau(Z)$ onto subtriples of $Z$. Then, quite like before, $\overline{\bigcup_n \Pi(T_n)}=Z$.
\end{proof}
\end{prop}

%

\begin{cor}
Every AF-triple system $Z$ is the norm closure of the linear span of its tripotents.
\end{cor}
In case of the principal part $P$ of $Z=P\oplus E\oplus S\oplus H$ this is a consequence of the above. For the other classes this follows by a more direct and quite easy approach. For the somewhat more difficult spin part $S$ use the fact \cite[p.358]{Harris-GeneralizationofCstar-algebras}
that there is always a scalar product as well as an involution $a\mapsto\bar{a}$ defined on $S$ so that
\begin{equation*}
\{a,b,c\}=\langle a,b\rangle c+\langle c,b\rangle a-\langle a,\bar{c}\rangle \bar{b}
\end{equation*}
as well as
\begin{equation*}
\|z\|^2=\langle z,z\rangle+\sqrt{\langle z,z\rangle^2-|\langle z,\bar{z}\rangle|^2}
\end{equation*}
for $a,b,c\in S$. Thus, in the self-adjoint part $S_{\text{sa}}$ all norm-one elements are tripotents.

\bibliographystyle{alpha}

\def\cprime{$'$} \def\cprime{$'$} \def\cprime{$'$}

\end{document}